\documentclass[11pt,reqno]{amsart}
\usepackage[all]{xy}
\usepackage{amsmath,amsfonts,amssymb,amsthm,mathrsfs,cleveref,setspace,enumitem,verbatim}
\usepackage{amsaddr}
\DeclareMathOperator*{\Sup}{sup}
\numberwithin{equation}{section}
\newtheorem{theorem}{Theorem}[section]

\newtheorem{example}[theorem]{Example}

\newtheorem{corollary}[theorem]{Corollary}

\newtheorem{lemma}[theorem]{Lemma}
\begin{document}
	\title[Complex Symmetry and Normality of Toeplitz Composition Operators]{Complex Symmetry and Normality of Toeplitz Composition Operators on the Hardy space}
	\author[\textbf{Anuradha Gupta} and \textbf{Aastha Malhotra}]
	{\textbf{Anuradha gupta$^1$} \and \textbf{Aastha Malhotra$^2$}}
	
	\address{$^1$Delhi College of Arts and Commerce,
		Department of Mathematics,\\
		University of Delhi,
		New Delhi-110023,
		India.\\
		E-mail address: dishna2@yahoo.in}
	
	\address{$^2$Faculty of Mathematical Sciences,
		Department of Mathematics,\\
		University of Delhi,
		Delhi-110007,
		India
		E-mail address: aasthamalhotra01@gmail.com}

	\begin{abstract}
		In this paper, we investigate the conditions under which the Toeplitz Composition operator on the Hardy space $\mathcal{H}^2$ becomes complex symmetric with respect to a certain conjugation. We also study various normality conditions for the Toeplitz Composition operator on $\mathcal{H}^2$.
	\end{abstract}
	\keywords{Complex symmetric, Composition operator, Conjugation, Toeplitz operator.}
	\subjclass[2010]{47B33,  47B35, 47B38}
	\maketitle
	\section{\bf{Introduction and Preliminaries}}
	Let $\mathbb{D}$ denote the open unit disc and $\mathbb{T }=\{e^{i\theta}:\theta\in[0,2\pi)\}$ denote the unit circle in the complex plane $\mathbb{C}$.  Recall that the \textit{Hardy space} $\mathcal{H}^2$ is a Hilbert space which consists of all those analytic functions $f$ on $\mathbb{D}$ having power series representation with square summable complex coefficients. That is,
	 \[\mathcal{H}^2=\{f:\mathbb{D}\rightarrow\mathbb{C}|f(z)=\sum_{n=0}^{\infty}\hat{f}(n)z^n \text{and\;} \|f\|_{\mathcal{H}^2}^2:=\sum_{n=0}^{\infty}|\hat{f}(n)|^2<\infty\} \] 
	 or equivalently,
	 \[\mathcal{H}^2=\{f:\mathbb{D}\rightarrow\mathbb{C}\; analytic|\Sup_{0<r<1}\frac{1}{2\pi}\int_{0}^{2\pi}|f(re^{i\theta})|^2 d\theta<\infty\}.\]
	 The evaluation of functions in $\mathcal{H}^2$ at each $w\in\mathbb{D}$ is a bounded linear functional and for all $f\in\mathcal{H}^2$, $f(w)=\langle f, K_w\rangle$ where $K_w(z)=1/(1-\overline{w}z)$. The function $K_w(z)$ is called the \textit{reproducing kernel} for the Hardy space $\mathcal{H}^2$.
	 Consider the Hilbert space
	 \[\widetilde{\mathcal{H}^2}=\{f^*:\mathbb{T}\rightarrow\mathbb{C}|f^*(z)=\sum_{n=0}^{\infty}\hat{f}(n)e^{in\theta} \text{and\;} \|f^*\|_{\mathcal{H}^2}^2:=\sum_{n=0}^{\infty}|\hat{f}(n)|^2<\infty\}. \]
	 Let $L^2$ denote the Lebesgue (Hilbert) space on the unit circle $\mathbb{T}$.
	 It is well known that every function $f\in\mathcal{H}^2$ satisfies the radial limit $f^*(e^{i\theta})=\lim\limits_{r\rightarrow 1^-}f(re^{i\theta})$ for almost every $\theta\in[0,2\pi)$ and it is obvious that the correspondence where  $f(z)=\sum_{n=0}^{\infty}\hat{f}(n)z^n$ is mapped to $f^*(e^{i\theta})=\sum_{n=0}^{\infty}\hat{f}(n)e^{in\theta} $ is an isometric isomorphism from $\mathcal{H}^2$ onto the closed subspace $\widetilde{\mathcal{H}^2}$ of $L^2$. Since $\{e_n(z)=z^n:n\in\mathbb{Z}\}$ forms an orthonormal basis for $L^2$, every function $f\in L^2$ can be expressed as $f(z)=\sum_{n=-\infty}^{\infty}\hat{f}(n)z^n$ where $\hat{f}(n)$ denotes the \textit{nth} fourier coefficient of $f$.  Let $L^\infty$ be the Banach space of all essentially bounded functions on the unit circle $\mathbb{T}$. For any $\phi\in L^{\infty}$, the \textit{Toeplitz operator} $T_\phi:\mathcal{H}^2\rightarrow\mathcal{H}^2$ is defined by $T_\phi f=P(\phi \cdot f)$ for $f\in\mathcal{H}^2$ where $P:L^2\rightarrow\mathcal{H}^2$ is the orthogonal projection. It can be easily verified that for $m,n\in\mathbb{Z}$,
	 \begin{equation*}
	 P(z^m\overline{z}^n)=
	 \begin{cases}
	 z^{m-n} & \text{if\;} m\geq n,\\
	 \quad 0 & \text{otherwise}.
	 \end{cases}
	 \end{equation*} 
	 
	 For a non-zero bounded analytic function $u$ on $\mathbb{D}$ and a self-analytic map $\phi$ on $\mathbb{D}$, the \textit{weighted composition operator} $W_{u,\phi}$ is defined by $W_{u,\phi}f=u\cdot f\circ\phi$ for every $f\in\mathcal{H}^2$. Over the past several decades, there has been tremendous development in the study of composition operators and weighted composition operators over the Hardy space $\mathcal{H}^2$ and various other spaces of analytic functions. Readers may refer \cite{MR1397026, MR1237406} for general study and background of the composition operators on the Hardy space $\mathcal{H}^2$. In this paper, we introduce the notion of the Toeplitz Composition operator on the Hardy space $\mathcal{H}^2$  where the symbol $u$ in $W_{u,\phi}$ need not necessarily be analytic. For a function $\psi\in L^\infty$ and a self-analytic map $\phi$ on $\mathbb{D}$, the \textit{Toeplitz Composition operator} $T_\psi C_\phi:\mathcal{H}^2\rightarrow\mathcal{H}^2$ is defined by $T_\psi C_\phi f= P(\psi \cdot f\circ \phi)$ for every $f\in\mathcal{H}^2$ where $C_\phi f:=f\circ \phi$ is the composition operator on $\mathcal{H}^2$. The authors in \cite{MR3941202} introduced the concept of the Toeplitz Composition operators on the Fock space and also studied its various properties.
	 
	 Let $\mathcal{H}$ be a separable Hilbert space. Then a mapping $S$ on $\mathcal{H}$ is said to be \textit{anti-linear (also conjugate-linear)} if $S(\alpha x_1+\beta x_2)=\overline{\alpha} S(x_1)+\overline{\beta} S(x_2)$ for all scalars $\alpha, \beta \in \mathbb{C}$ and for all $x_1, x_2\in \mathcal{H}$.
	 
	 An anti-linear mapping $\mathcal{C}:\mathcal{H}\rightarrow\mathcal{H}$ is said to be a \textit{conjugation} if it is involutive (\textit{i.e.} $\mathcal{C}^2=I$) and isometric (\textit{i.e.} $\|\mathcal{C}x\|=\|x\|$ for every $x\in\mathcal{H}$). A \textit{complex symmetric operator} $S$ on $\mathcal{H}$ is a bounded linear operator such that $S=\mathcal{C}S^*\mathcal{C}$ for some conjugation $\mathcal{C}$ on $\mathcal{H}$. We call such an operator $S$ to be a \textit{$\mathcal{C}-$symmetric operator}.
	 
	Garcia and Putinar \cite{MR2187654,MR2302518} began the general study of complex symmetric operators on Hilbert spaces which are the natural generalizations of complex symmetric matrices. There exist a wide variety of complex symmetric operators which include normal operators, compressed Toeplitz operators, Volterra integration operators \textit{etc.} Jung \textit{et al.}\cite{MR3210031} studied the complex symmetry of the weighted composition operators on the Hardy space in the unit disc $\mathbb{D}$. Garcia and Hammond \cite{MR3203059} undertook the study of complex symmetry of weighted composition operators on the weighted Hardy spaces. Ko and Lee \cite{MR3404546} gave a characterization of the complex symmetric Toeplitz operators on the Hardy space $\mathcal{H}^2$ of the unit disc $\mathbb{D}$. Motivated by this, we study the complex symmetry of the Toeplitz Composition operators on the Hardy space $\mathcal{H}^2$. In this paper we give a characterization of such types of operators. We also investigate certain conditions under which a complex symmetric operator turns out to be a normal operator. In the concluding section of this article, we discuss the normality of the Toeplitz Composition operators on $\mathcal{H}^2$.
	
	\section{\bf{Complex Symmetric Toeplitz Composition Operators}}
	In this section we aim to find the conditions under which a Toeplitz Composition operator becomes complex symmetric with respect to a certain fixed conjugation. In order to determine these conditions, we need an explicit formula for the adjoint $C_\phi^*$ of a composition operator $C_\phi$ where $\phi$ is a self-analytic map on the unit disc $\mathbb{D}$. But there exists no general formula and there are only a few special cases where it is possible to find a formula for $C_\phi^*$ explicitly. C. Cowen was the first to find the representation for the adjoint of a composition operator $C_\phi$ on $\mathcal{H}^2$, famously known as the Cowen's Adjoint Formula, where the symbol $\phi$ is a linear fractional self-map of the unit disc $\mathbb{D}$. 
	The Cowen's adjoint formula was extended to the Bergman space $\mathcal{A}^2$ by P.Hurst \cite{MR1444662} and it is stated as follows:
	\begin{theorem}[\cite{MR1397026}]{(Cowen's Adjoint Formula)}\label{CAF}
		Let $\phi(z)=\frac{az+b}{cz+d}$ be a linear fractional self-map of the unit disc where $ad-bc\neq0$. Then $\sigma(z)=\frac{\overline{a}z-\overline{c}}{\overline{-b}z+\overline{d}}$ maps disc into itself, $g(z)=(\overline{-b}z+\overline{d})^{-p}$ and $h(z)=(cz+d)^p$ are bounded analytic functions on the disc and on $\mathcal{H}^2$ or $\mathcal{A}^2$, $C_\phi^*=M_gC_\sigma M_h^*$ where $p=1$ on $\mathcal{H}^2$ and $p=2$ on $\mathcal{A}^2$.(Note that the operator $M_g$ is the multiplication operator defined by $M_g f=g\cdot f$.)
	\end{theorem}
Next we have the following lemmas which would be instrumental in proving certain results throughout this article :
	\begin{lemma}[\cite{MR2253017}]\label{LFT}
		A linear fractional map $\phi$, written in the form $\phi(z)=\frac{az+b}{cz+d}$;  $ad-bc\neq0$, maps $\mathbb{D}$ into itself if and only if:
		\begin{equation} 
		|b\overline{d}-a\overline{c}|+|ad-bc|\leq|d|^2-|c|^2.
		\end{equation}
	\end{lemma}

\begin{lemma}[\cite{MR1397026}]
	Let $\phi(z)=\frac{az+b}{cz+d}$ be a linear fractional map and define the associated linear fractional transformation $\phi^*$ by
	\[\phi^*(z)=\frac{1}{\overline{\phi^{-1}(\frac{1}{\overline{z}})}}=\frac{\overline{a}z-\overline{c}}{\overline{-b}z+\overline{d}}.\]
	Then $\phi$ is a self-map of the disc if and only if $\phi^*$ is also a self-map of the disc.
\end{lemma}

\begin{lemma}[\cite{MR1397026}]
	If $\phi(z)=\frac{az+b}{cz+d}$ is a linear fractional transformation mapping $\mathbb{D}$ into itself where $ad-bc=1$, then $\sigma(z)=\frac{\overline{a}z-\overline{c}}{\overline{-b}z+\overline{d}}$ maps  $\mathbb{D}$ into itself.
 \end{lemma}

In the following lemma, a conjugation on the Hardy space $\mathcal{H}^2$ has been defined with respect to which we will find the complex symmetry of the operator $T_\psi C_\phi$.
\begin{lemma}[\cite{MR3404546}]
	For every $\xi$ and $\theta$, let $C_{\xi,\theta}:\mathcal{H}^2\rightarrow \mathcal{H}^2$ be defined by 
	\[C_{\xi,\theta}f(z)=e^{i\xi} \overline{f(e^{i\theta}\overline{z})}.\]
	Then $C_{\xi,\theta}$ is a conjugation on $\mathcal{H}^2$. Moreover, $C_{\xi,\theta}$ and $C_{\tilde{\xi},\tilde{\theta}}$ are unitarily equivalent where $(\tilde{\xi},\tilde{\theta})$ satisfies the equation $\tilde{\xi}-k\tilde{\theta}=-\xi+k\theta-2n\pi$ for every $k\in\mathbb{N}$ and $n\in\mathbb{Z}$.
\end{lemma}
In the next theorem, we determine the conditions under which the Toeplitz Composition operator $T_\psi C_\phi$ turns out to be complex symmetric with respect to the conjugation $C_{\xi,\theta}$ on $\mathcal{H}^2$.
\begin{theorem}
	For $\psi(z)=\sum_{n=-\infty}^{\infty}\hat{\psi}(n)z^n \in L^\infty$ and for self-analytic linear transformation $\phi(z)=az+b\;(a\neq0)$ mapping $\mathbb{D}$ into itself, let $T_{\psi}C_{\phi}$ be a Toeplitz Composition operator on  $\mathcal{H}^2$. Then $T_{\psi}C_{\phi}$ is complex symmetric with the conjugation $C_{\xi,\theta}$ if and only if for each $k, p\in \mathbb{N}\cup\{0\}$, we have $\sum_{n=-k+p}^{p} \binom{k}{p-n}\overline{\hat{\psi}(n)}\overline{a}^{p-n}\overline{b}^{n+k-p}\lambda^p=\sum_{n=-k}^{-k+p} \binom{p}{p-n-k}\overline{\hat{\psi}(-n)}\overline{a}^{n+k}\overline{b}^{p-n-k}\lambda^k$. 
\end{theorem}
\begin{proof}
	If $T_\psi C_\phi$ is complex symmetric with respect to the conjugation  $C_{\xi,\theta}$, then for all $k\in \mathbb{N}\cup\{0\}$ we have
	\begin{equation}\label{CSC}
	C_{\xi,\theta}T_\psi C_\phi z^k=(T_\psi C_\phi)^*	C_{\xi,\theta}z^k.
	\end{equation}
	If we take $\mu=e^{i\xi}$ and $\lambda=e^{-i\theta}$, then we have
	\begin{align}
	C_{\xi,\theta}T_\psi C_\phi z^k &=\notag C_{\xi,\theta}T_\psi (\phi(z))^k\\
	&=\notag C_{\xi,\theta}T_\psi(az+b)^k\\
	&= \notag C_{\xi,\theta}P(\psi(z)\cdot\sum_{m=0}^{k}\binom{k}{m}a^mb^{k-m}z^m)\\
	&=\notag C_{\xi,\theta}P(\sum_{m=0}^{k}(\sum_{n=-\infty}^{\infty}\binom{k}{m}\hat{\psi}(n) a^mb^{k-m}z^{m+n}))\\
	&=\notag C_{\xi,\theta}(\sum_{m=0}^{k}P(\sum_{n=-\infty}^{\infty}\binom{k}{m}\hat{\psi}(n) a^mb^{k-m}z^{m+n}))\\
	&=\notag C_{\xi,\theta}(\sum_{m=0}^{k}(\sum_{n=-m}^{\infty}\binom{k}{m}\hat{\psi}(n) a^mb^{k-m}z^{m+n}))\\
	&=\notag \sum_{m=0}^{k}C_{\xi,\theta}(\sum_{n=-m}^{\infty}\binom{k}{m}\hat{\psi}(n) a^mb^{k-m}z^{m+n})\\
	&=\notag e^{i\xi}\sum_{m=0}^{k}(\sum_{n=-m}^{\infty}\binom{k}{m}\overline{\hat{\psi}(n)} \overline{a}^m\overline{b}^{k-m}e^{-i(m+n)\theta}z^{m+n})\\
	&=\label{X}\mu \sum_{m=0}^{k}(\sum_{n=-m}^{\infty}\binom{k}{m}\overline{\hat{\psi}(n)} \overline{a}^m\overline{b}^{k-m}\lambda^{m+n}z^{m+n})
	\end{align} 
	and 
	\begin{align}
	\notag(T_\psi C_\phi)^*C_{\xi,\theta}z^k &=\notag C_\phi^*T_\psi^*C_{\xi,\theta}z^k\\
	&=\notag C_\phi^*T_{\overline{\psi}}(e^{i\xi}e^{-ik\theta}z^k)\\
	&=\notag C_\phi^*T_{\overline{\psi}}(\mu\lambda^kz^k)\\
	&=\notag C_\phi^*P(\mu\lambda^k\sum_{n=-\infty}^{\infty}\overline{\hat{\psi}(n)}z^{k-n})\\
	&=\notag C_\phi^*P(\mu\lambda^k\sum_{n=-\infty}^{\infty}\overline{\hat{\psi}(-n)}z^{n+k})\\
	&=\label{C}\mu\lambda^kC_\phi^*(\sum_{n=-k }^{\infty}\overline{\hat{\psi}(-n)}z^{n+k}).
	\end{align}
	On using theorem \ref{CAF} for $a\neq0$, $c=0$ and $d=1$, we obtain that $C_\phi^*=M_gC_\sigma$ where $g(z)=(1-\overline{b}z)^{-1}$ and $\sigma(z)=\frac{\overline{a}z}{1-\overline{b}z}$. Since $|a|+|b|\leq1$ from lemma \ref{LFT}, so $|b|<1$ and hence, $\frac{1}{(1-\overline{b}z)^i}=\sum_{j=0}^{\infty}\binom{j+i-1}{j}(\overline{b}z)^j$ for $z\in \mathbb{D}$. Therefore, from \eqref{C} we get that 
	\begin{align}
	(T_\psi C_\phi)^*C_{\xi,\theta}z^k&=\notag \mu\lambda^kM_gC_\sigma(\sum_{n=-k }^{\infty}\overline{\hat{\psi}(-n)}z^{n+k})\\
	&= \notag\mu\lambda^kM_g(\sum_{n=-k }^{\infty}\overline{\hat{\psi}(-n)}\left(\frac{\overline{a}z}{1-\overline{b}z}\right)^{n+k})\\
	&=\notag \mu\lambda^k(\sum_{n=-k }^{\infty}\overline{\hat{\psi}(-n)}\overline{a}^{n+k}\left(\frac{1}{1-\overline{b}z}\right)^{n+k+1}z^{n+k})\\
	&=\label{Y}\mu\sum_{j=0}^{\infty}(\sum_{n=-k }^{\infty}\binom{n+k+j}{j}\overline{\hat{\psi}(-n)}\overline{a}^{n+k}\overline{b}^j\lambda^kz^{n+k+j}).
	\end{align}
	It follows from \eqref{CSC} that for each $k\in\mathbb{N}\cup\{0\}$, we have
	\begin{equation}\label{D}
	 \sum_{m=0}^{k}(\sum_{n=-m}^{\infty}\binom{k}{m}\overline{\hat{\psi}(n)} \overline{a}^m\overline{b}^{k-m}\lambda^{m+n}z^{m+n})=\sum_{j=0}^{\infty}(\sum_{n=-k }^{\infty}\binom{n+k+j}{j}\overline{\hat{\psi}(-n)}\overline{a}^{n+k}\overline{b}^j\lambda^kz^{n+k+j}).
	\end{equation}
	Thus, the coefficient of $z^p$ where $p\in\mathbb{N}\cup\{0\}$ must be equal on the both sides of \eqref{D} and we deduce that
	\begin{equation}\label{E}
\sum_{n=-k+p}^{p} \binom{k}{p-n}\overline{\hat{\psi}(n)}\overline{a}^{p-n}\overline{b}^{n+k-p}\lambda^p=\sum_{n=-k}^{-k+p} \binom{p}{p-n-k}\overline{\hat{\psi}(-n)}\overline{a}^{n+k}\overline{b}^{p-n-k}\lambda^k
	\end{equation}
	for each $k\in\mathbb{N}\cup\{0\}$.
	
	Conversely, let us suppose that \eqref{E} holds for each $k, p\in\mathbb{N}\cup\{0\}$. Then from \eqref{X} and \eqref{Y}, we have
	\begin{align*}
	(C_{\xi,\theta}T_\psi C_\phi-(T_\psi C_\phi)^*C_{\xi,\theta})z^k &=\mu (\sum_{m=0}^{k}(\sum_{n=-m}^{\infty}\binom{k}{m}\overline{\hat{\psi}(n)} \overline{a}^m\overline{b}^{k-m}\lambda^{m+n}z^{m+n}))\\ &-\mu (\sum_{j=0}^{\infty}(\sum_{n=-k }^{\infty}\binom{n+k+j}{j}\overline{\hat{\psi}(-n)}\overline{a}^{n+k}\overline{b}^j\lambda^kz^{n+k+j}))\\&=0.
	\end{align*}     
	Thus, $T_\psi C_\phi$ is complex symmetric with conjugation $ C_{\xi,\theta}$.
\end{proof}
\begin{example}
	Let $\psi(z)=z+\overline{z}\in L^\infty$. Then, $\hat{\psi}(n)=\hat{\psi}(-n)$ for all $n\in\mathbb{Z}$. Let $\phi(z)=iz$. Then $\phi(z)$ is a self-analytic map on $\mathbb{D}$. Consider the conjugation $C_{\xi,\theta}$ where we choose $\theta=\pi/2$. Then $\lambda=e^{-i\theta}=-i$. On substituting $a=i$, $b=0$ and $\lambda=-i$ in \eqref{X} and \eqref{Y}, we have
	\begin{align*}
	C_{\xi,\theta}T_\psi C_\phi z^k&=\mu\sum_{n=-k}^{\infty}\hat{\psi}(n)\overline{a}^k\lambda^{n+k}z^{n+k}\\
	&=\mu\sum_{n=-k}^{\infty}\hat{\psi}(n)(-i)^{n+2k}z^{n+k}\\
	&=\mu\sum_{n=-k}^{\infty}\hat{\psi}(-n)(-i)^{n+2k}z^{n+k}\\
	&=\mu\sum_{n=-k}^{\infty}\hat{\psi}(-n)\overline{a}^{n+k}\lambda^{k}z^{n+k}\\
	&=(T_\psi C_\phi)^*C_{\xi,\theta}z^k
	 \end{align*}
	for each $k\in\mathbb{N}\cup\{0\}$.
	Hence, the operator $T_\psi C_\phi$ is complex symmetric with respect to the conjugation $C_{\xi,\pi/2}$.
\end{example}
In the light of the above example, an interesting observation has been made in the following Corollary:
\begin{corollary}
	Let $\psi(z)=\sum_{n=-\infty}^{\infty}\hat{\psi}(n)z^n \in L^\infty$ be such that $\hat{\psi}(n)=\hat{\psi}(-n)$ for all $n\in\mathbb{Z}$ and suppose that $\phi(z)=az$ is a self-analytic map on $\mathbb{D}$ where $a=e^{i\theta}$ for $\theta\in\mathbb{R}$. Then $T_\psi C_\phi$ is complex symmetric with respect to the conjugation $C_{\xi,\theta}$.
\end{corollary}
\begin{proof}
	On using \eqref{X} and \eqref{Y} where $a=e^{i\theta}$, $b=0$ and $\lambda=e^{-i\theta}$, we obtain that for each $k\in\mathbb{N}\cup\{0\}$,
	\begin{align*}
	C_{\xi,\theta}T_\psi C_\phi z^k&=\mu \sum_{m=0}^{k}(\sum_{n=-m}^{\infty}\binom{k}{m}\overline{\hat{\psi}(n)} \overline{a}^m\overline{b}^{k-m}\lambda^{m+n}z^{m+n})\\
	&=\mu\sum_{n=-k}^{\infty}\overline{\hat{\psi}(n)}\overline{a}^k\lambda^{n+k}z^{n+k}\\
	&=\mu\sum_{n=-k}^{\infty}\overline{\hat{\psi}(n)}e^{-i(n+2k)\theta}z^{n+k}\\
	&=\mu\sum_{n=-k}^{\infty}\overline{\hat{\psi}(-n)}e^{-i(n+2k)\theta}z^{n+k}\\
	&=\mu\sum_{n=-k }^{\infty}\overline{\hat{\psi}(-n)}\overline{a}^{n+k}\lambda^kz^{n+k}\\
	&=(T_\psi C_\phi)^*C_{\xi,\theta}z^k.
	\end{align*}
	Hence, $ T_\psi C_\phi$ is complex symmetric with conjugation $C_{\xi,\theta}$ where $a=e^{i\theta}(\theta\in\mathbb{R})$.
\end{proof}
 
An operator $T:\mathcal{H}\rightarrow\mathcal{H}$ where $\mathcal{H}$ denotes a Hilbert space is said to be \textit{hyponormal} if $T^*T\geq TT^*$ or equivalently, $\|Tx\|\geq \|T^*x\|$ for every $x\in \mathcal{H}$.
Our next goal is to find out the conditions under which a Toeplitz Composition operator $T_\psi C_\phi$ becomes a normal operator. The proof involves the technique followed in [Proposition 2.2, \cite{MR3203059}].
\begin{theorem}
	Let $\psi \in L^\infty$ and let $\phi$ be any self-analytic mapping from $\mathbb{D}$ into itself. If the operator $T_\psi C_\phi:\mathcal{H}^2\rightarrow \mathcal{H}^2$ is hyponormal and complex symmetric with conjugation $C_{\xi,\theta}$, then $T_\psi C_\phi$ is a normal operator on $\mathcal{H}^2$.
\end{theorem}
\begin{proof}
	 Since $T_\psi C_\phi$ is complex symmetric with respect to the conjugation $C_{\xi,\theta}$, this gives that $(T_\psi C_\phi)^*=C_{\xi,\theta}T_\psi C_\phi C_{\xi,\theta}$. On using the isometry of $C_{\xi,\theta}$, we obtain that $\|(T_\psi C_\phi)^*f\|=\|C_{\xi,\theta}T_\psi C_\phi C_{\xi,\theta}f\|=\|T_\psi C_\phi C_{\xi,\theta}f\|$ for every $f\in\mathcal{H}^2$.
	 
	  By hypothesis, $T_\psi C_\phi$ is a hyponormal operator on $\mathcal{H}^2$ and thus, $\|T_\psi C_\phi f\|\geq\|(T_\psi C_\phi)^*f\|$ for every $f\in\mathcal{H}^2$. Therefore, $\|(T_\psi C_\phi)^*f\|=\|T_\psi C_\phi C_{\xi,\theta}f\|\geq\\ \|(T_\psi C_\phi)^* C_{\xi,\theta}f\|=\|C_{\xi,\theta}T_\psi C_\phi f\|=\|T_\psi C_\phi f\|$ for every $f\in\mathcal{H}^2$. Hence,  $\|(T_\psi C_\phi)^*f\|\\ \geq \|T_\psi C_\phi f\|$ and this together with the hyponormality of $T_\psi C_\phi$ implies that $\|(T_\psi C_\phi)^*f\|\\ =\|T_\psi C_\phi f\|$ for every $f\in\mathcal{H}^2$ which proves that $T_\psi C_\phi$ is a normal operator.
\end{proof} 
In the following lemma, the conditions under which the Toeplitz Composition operator $T_\psi C_\phi$ commutes with the conjugation $C_{\xi,\theta}$ has been investigated which further provides us with a criteria which together with the complex symmetry of $T_\psi C_\phi$ makes the operator $T_\psi C_\phi$ a normal operator.
\begin{theorem}
		Let $\psi(z)=\sum_{n=-\infty}^{\infty}\hat{\psi}(n)z^n \in L^\infty$ and $\phi(z)=az+b\;(a\neq0)$ be a linear fractional transformation mapping $\mathbb{D}$ into itself. Then the Toeplitz Composition operator $T_{\psi}C_{\phi}$ commutes with the conjugation $C_{\xi,\theta}$ on $\mathcal{H}^2$ if and only if $\hat{\psi}(n)a^mb^{k-m}\lambda^k=\overline{\hat{\psi}(n)}\overline{a}^m\overline{b}^{k-m}\lambda^{m+n}$ for each $n\in \mathbb{Z}$ and $m, k\in\mathbb{N}\cup\{0\}(0\leq m\leq k)$. 
	\end{theorem}
\begin{proof}
	If the operator $T_\psi C_\phi$ commutes with $C_{\xi,\theta}$, then for each $k\in\mathbb{N}\cup\{0\}$, we have $T_\psi C_\phi C_{\xi,\theta}z^k=C_{\xi,\theta}T_\psi C_\phi z^k$.
	Since for each $k\in\mathbb{N}\cup\{0\}$,
	\begin{align*}
	T_\psi C_\phi C_{\xi,\theta}z^k &=T_\psi C_\phi(e^{i\xi}e^{-ik\theta}z^k)\\
	&=P(\psi(z)\cdot\mu\lambda^k(az+b)^k)\\
	&=\mu\lambda^kP(\sum_{m=0}^{k}(\sum_{n=-\infty}^{\infty}\binom{k}{m}\hat{\psi}(n) a^mb^{k-m}z^{m+n}))\\
	&=\mu\lambda^k(\sum_{m=0}^{k}P(\sum_{n=-\infty}^{\infty}\binom{k}{m}\hat{\psi}(n) a^mb^{k-m}z^{m+n}))\\
	&=\mu\lambda^k\sum_{m=0}^{k}(\sum_{n=-m}^{\infty}\binom{k}{m}\hat{\psi}(n) a^mb^{k-m}z^{m+n})
	\end{align*}
	and
	\begin{align*}
	C_{\xi,\theta}T_\psi C_\phi z^k &=C_{\xi,\theta}T_\psi((az+b)^k) \\
	&=C_{\xi,\theta}P(\sum_{m=0}^{k}(\sum_{n=-\infty}^{\infty}\binom{k}{m}\hat{\psi}(n) a^mb^{k-m}z^{m+n}))\\
	&=C_{\xi,\theta}(\sum_{m=0}^{k}P(\sum_{n=-\infty}^{\infty}\binom{k}{m}\hat{\psi}(n) a^mb^{k-m}z^{m+n}))\\
	&=C_{\xi,\theta}(\sum_{m=0}^{k}(\sum_{n=-m}^{\infty}\binom{k}{m}\hat{\psi}(n) a^mb^{k-m}z^{m+n}))\\
	&=\mu\sum_{m=0}^{k}(\sum_{n=-m}^{\infty}\binom{k}{m}\overline{\hat{\psi}(n)}\overline{a}^m\overline{b}^{k-m}\lambda^{m+n}z^{m+n})\; ;
	\end{align*}
	we obtain that $\hat{\psi}(n)a^mb^{k-m}\lambda^k=\overline{\hat{\psi}(n)}\overline{a}^m\overline{b}^{k-m}\lambda^{m+n}$ for each $n\in \mathbb{Z}$ and $m\in\mathbb{N}\cup\{0\}(0\leq m\leq k)$.
	
	Conversely, if for each $n\in \mathbb{Z}$ and $m, k\in\mathbb{N}\cup\{0\}$, we have $\hat{\psi}(n)a^mb^{k-m}\lambda^k=\overline{\hat{\psi}(n)}\overline{a}^m\overline{b}^{k-m}\lambda^{m+n}$, then $(	T_\psi C_\phi C_{\xi,\theta}-C_{\xi,\theta}T_\psi C_\phi) z^k=0$ which proves that $T_{\psi}C_{\phi}$ commutes with $C_{\xi,\theta}$.
\end{proof}
\begin{corollary}\label{COO}
	Let $\psi(z)=\sum_{n=-\infty}^{\infty}\hat{\psi}(n)z^n \in L^\infty$ and $\phi(z)=az+b\;(a\neq0)$ be a linear fractional transformation mapping $\mathbb{D}$ into itself. Then the Toeplitz Composition operator $T_{\psi}C_{\phi}$ commutes with the conjugation $C_{0,0}$ on $\mathcal{H}^2$ if and only if $\hat{\psi}(n)a^mb^{k-m}\in\mathbb{R}$ for each $n\in \mathbb{Z}$ and $m, k\in\mathbb{N}\cup\{0\}(0\leq m\leq k)$. 
\end{corollary}
The following theorem is in general valid for any linear operator $T$ on a Hilbert space $\mathcal{H}$ which is complex symmetric with respect to any conjugation $\mathcal{C}$ defined on $\mathcal{H}$ such that $T$ commutes with $\mathcal{C}$.
\begin{theorem}\label{commute}
	Let $\psi \in L^\infty$ and let $\phi$ be any self-analytic mapping from $\mathbb{D}$ into itself. Suppose that $T_\psi C_\phi $ is a complex symmetric operator with conjugation $C_{\xi,\theta}$ on $\mathcal{H}^2$ and further, suppose that $T_{\psi}C_{\phi}$ commutes with $C_{\xi,\theta}$. Then $T_\psi C_\phi$ is a normal operator on $\mathcal{H}^2$.
\end{theorem}
\begin{proof}
	By hypothesis, $T_\psi C_\phi$ is a complex symmetric operator with conjugation $C_{\xi,\theta}$ such that it commutes with $C_{\xi,\theta}$  which implies that  $T_\psi C_\phi$ is a self-adjoint operator. That is,
	\begin{equation}\label{B1}
(T_\psi C_\phi)^*=C_{\xi,\theta}T_\psi C_\phi C_{\xi,\theta}=C_{\xi,\theta}C_{\xi,\theta}T_\psi C_\phi=T_\psi C_\phi.
	\end{equation}
Hence, $T_\psi C_\phi$ is a normal operator on  $\mathcal{H}^2$.
\end{proof}
\begin{corollary}
	Let $\psi(z)=\sum_{n=-\infty}^{\infty}\hat{\psi}(n)z^n \in L^\infty$ and $\phi(z)=az+b\;(a\neq0)$ be a linear fractional transformation mapping $\mathbb{D}$ into itself. Suppose that $T_\psi C_\phi:\mathcal{H}^2\rightarrow \mathcal{H}^2$ is a complex symmetric operator with conjugation $C_{0,0}$ and further suppose that $\hat{\psi}(n)a^mb^{k-m}\in\mathbb{R}$ for each $n\in \mathbb{Z}$ and $m, k\in\mathbb{N}\cup\{0\}(0\leq m\leq k)$. Then $T_\psi C_\phi$ is a normal operator on $\mathcal{H}^2$.
	\end{corollary}
\begin{proof}
	From Corollary \ref{COO}, we obtain that $T_{\psi}C_{\phi}$ commutes with the conjugation $C_{0,0}$ as $\hat{\psi}(n)a^mb^{k-m}\in\mathbb{R}$ for each $n\in \mathbb{Z}$ and $m, k\in\mathbb{N}\cup\{0\}(0\leq m\leq k)$. Thus,  we get that $T_\psi C_\phi$ is a normal operator on $\mathcal{H}^2$ by Theorem \ref{commute}.
\end{proof}
\section{Normality Of Toeplitz Composition Operators}

In this section we discuss the normality of the Toeplitz Composition operators on $\mathcal{H}^2$. We explore the conditions under which the operator $T_{\psi}C_{\phi}$ becomes normal and further we discover the necessary and sufficient conditions for the operator $T_{\psi}C_{\phi}$ to be Hermitian.
\begin{theorem}
	Let $\psi(z)=\sum_{n=-\infty}^{\infty}\hat{\psi}(n)z^n \in L^\infty$ and $\phi(z)=az+b\;(a\neq0)$ be a linear fractional transformation mapping $\mathbb{D}$ into itself. If the operator $T_\psi C_\phi$ on $\mathcal{H}^2$ is hyponormal, then $\sum_{n=0}^{\infty}\{|\hat{\psi}(n)|^2-\sum_{m=0}^{\infty}(\binom{m+n}{m}|\hat{\psi}(-n)| |a|^{n}|b|^{m})^2\}\geq0$.
\end{theorem}

\begin{proof}
	By the hyponormality of $T_\psi C_\phi$ on $\mathcal{H}^2$, we have $\|T_\psi C_\phi f \|^2\geq\|(T_\psi C_\phi)^* f \|^2$  for every $f\in\mathcal{H}^2$. In particular, on taking $f\equiv1$, we obtain that
	\begin{equation}\label{N1}
	\|T_\psi C_\phi (1) \|^2\geq\|(T_\psi C_\phi)^* (1) \|^2.
	\end{equation} 
	Then $\|T_\psi C_\phi (1) \|^2=\|P(\sum_{n=-\infty}^{\infty}\hat{\psi}(n)z^n)\|^2=\|\sum_{n=0}^{\infty}\hat{\psi}(n)z^n\|^2=\sum_{n=0}^{\infty}|\hat{\psi}(n)|^2$. It can be noted that the function $\psi(z)$ can be expressed as\[\psi(z)=\psi_{+}(z)+\psi_{0}(z)+\overline{\psi_{-}(z)}\]
	where $\psi_{+}(z)=\sum_{n=1}^{\infty}\hat{\psi}(n)z^n$, $\psi_{-}(z)=\sum_{n=1}^{\infty}\overline{\hat{\psi}(-n)}z^n$ and $\psi_{0}(z)=\hat{\psi}(0)$.
	It follows that $P(\overline{\psi(z)})=P(\overline{\psi_{+}(z)}+\overline{\psi_{0}(z)}+\psi_{-}(z))=\sum_{n=0}^{\infty}\overline{\hat{\psi}(-n)}z^n$.
Since $C_\phi^*=M_gC_\sigma$ where $g(z)=(1-\overline{b}z)^{-1}$ and $\sigma(z)=\frac{\overline{a}z}{1-\overline{b}z}$, it is obtained that \begin{align*}
\|(T_\psi C_\phi)^* (1) \|^2=\|C_\phi^*T_{\overline{\psi}}(1)\|^2&=\|M_gC_\sigma P(\overline{\psi(z)})\|^2\\
&=\|M_gC_\sigma (\sum_{n=0}^{\infty}\overline{\hat{\psi}(-n)}z^n)\|^2\\
&= \|\sum_{n=0}^{\infty}\overline{\hat{\psi}(-n)}\frac{\overline{a}^{n}z^n}{(1-\overline{b}z)^{n+1}}\|^2\\
&=\|\sum_{n=0}^{\infty}(\sum_{m=0}^{\infty}\binom{m+n}{m}\overline{\hat{\psi}(-n)} \overline{a}^{n}\overline{b}^{m}z^{m+n})\|^2\\
&=\sum_{n=0}^{\infty}\left(\sum_{m=0}^{\infty}\left(\binom{m+n}{m}|\hat{\psi}(-n)| |a|^{n}|b|^{m}\right)^2\right).
\end{align*} 
Hence, it follows from \eqref{N1} that  $\sum_{n=0}^{\infty}\{|\hat{\psi}(n)|^2-\sum_{m=0}^{\infty}(\binom{m+n}{m}|\hat{\psi}(-n)| |a|^{n}|b|^{m})^2\}\geq0$.
\end{proof}
\begin{corollary}\label{N2}
	Let $\psi(z)=\sum_{n=-\infty}^{\infty}\hat{\psi}(n)z^n \in L^\infty$ and $\phi(z)=az+b\;(a\neq0)$ be a linear fractional transformation mapping $\mathbb{D}$ into itself. If the operator $T_\psi C_\phi$ on $\mathcal{H}^2$ is normal, then $\sum_{n=0}^{\infty}\{|\hat{\psi}(n)|^2-\sum_{m=0}^{\infty}(\binom{m+n}{m}|\hat{\psi}(-n)| |a|^{n}|b|^{m})^2\}=0$.
	\end{corollary}
The condition obtained above in Corollary \ref{N2} is necessary but not sufficient which can be observed through the following example:

\begin{example}
Let $\psi(z)=z+\overline{z}$ and $\phi(z)=iz$. Then, for $a=i$, $b=0$, $\hat{\psi}(-1)=\hat{\psi}(1)=1$ and $\hat{\psi}(-n)=\hat{\psi}(n)=0$ where $n\in\mathbb{Z}-\{0\}$, the condition  $\sum_{n=0}^{\infty}\{|\hat{\psi}(n)|^2-\sum_{m=0}^{\infty}(\binom{m+n}{m}|\hat{\psi}(-n)| |a|^{n}|b|^{m})^2\}=0$  is satisfied. But the Toeplitz Composition operator $T_\psi C_\phi$ is not normal as $(T_\psi C_\phi)(T_\psi C_\phi)^*(z)=z^3+2z $ whereas $(T_\psi C_\phi)^*(T_\psi C_\phi)(z)=-z^3+2z$.
\end{example}
Next we investigate the necessary and sufficient conditions under which the operator $T_\psi C_\phi$ becomes Hermitian.
\begin{theorem}
	Let $\psi(z)=\sum_{n=-\infty}^{\infty}\hat{\psi}(n)z^n \in L^\infty$ and $\phi(z)=az+b\;(a\neq0)$ be a linear fractional transformation mapping $\mathbb{D}$ into itself. Then the Toeplitz Composition operator $T_{\psi}C_{\phi}$ on $\mathcal{H}^2$ is Hermitian if and only if for each $k, p\in \mathbb{N}\cup\{0\}$, we have $\sum_{n=-k+p}^{p} \binom{k}{p-n}\hat{\psi}(n)a^{p-n}b^{n+k-p}=\sum_{n=-k}^{-k+p} \binom{p}{p-n-k}\overline{\hat{\psi}(-n)}\overline{a}^{n+k}\overline{b}^{p-n-k}$. 
\end{theorem}
\begin{proof}
	Let us suppose that the operator $T_{\psi}C_{\phi}$ is Hermitian on $\mathcal{H}^2$. This implies that $T_\psi C_\phi z^k=(T_\psi C_\phi)^*z^k$ for every $k\in \mathbb{N}\cup\{0\}$. Since 
	\begin{align*}
	T_\psi C_\phi z^k & =T_\psi (\phi(z))^k\\
	&=P(\psi(z)\cdot \sum_{m=0}^{k}\binom{k}{m}a^mb^{k-m}z^m)\\
	&=P(\sum_{m=0}^{k}(\sum_{n=-\infty}^{\infty}\binom{k}{m}\hat{\psi}(n) a^mb^{k-m}z^{m+n}))\\
	&=\sum_{m=0}^{k}P(\sum_{n=-\infty}^{\infty}\binom{k}{m}\hat{\psi}(n) a^mb^{k-m}z^{m+n})\\
	&=\sum_{m=0}^{k}(\sum_{n=-m}^{\infty}\binom{k}{m}\hat{\psi}(n) a^mb^{k-m}z^{m+n})
	\end{align*}
	and
	\begin{align*}
	(T_\psi C_\phi)^*z^k &=C_\phi^*T_{\overline{\psi}}z^k\\
	&=C_\phi^*P(\sum_{n=-\infty}^{\infty}\overline{\hat{\psi}(-n)}z^{n+k})\\
	&=M_gC_\sigma(\sum_{n=-k }^{\infty}\overline{\hat{\psi}(-n)}z^{n+k})\\
	&=\sum_{n=-k }^{\infty}\overline{\hat{\psi}(-n)}\overline{a}^{n+k}\left(\frac{1}{1-\overline{b}z}\right)^{n+k+1}z^{n+k})\\
	&=\sum_{j=0}^{\infty}(\sum_{n=-k }^{\infty}\binom{n+k+j}{j}\overline{\hat{\psi}(-n)}\overline{a}^{n+k}\overline{b}^j\lambda^kz^{n+k+j})
	\end{align*}
	where $g(z)=(1-\overline{b}z)^{-1}$ and $\sigma(z)=\frac{\overline{a}z}{1-\overline{b}z}$; it follows that the coefficient of $z^p$ for $p\in \mathbb{N}\cup\{0\}$ in the expressions for $T_\psi C_\phi z^k$ and $(T_\psi C_\phi)^*z^k$ are equal for each $k\in \mathbb{N}\cup\{0\}$. Therefore, for each $k, p\in \mathbb{N}\cup\{0\}$, we have 
	\begin{equation}\label{A1}
	\sum_{n=-k+p}^{p} \binom{k}{p-n}\hat{\psi}(n)a^{p-n}b^{n+k-p}=\sum_{n=-k}^{-k+p} \binom{p}{p-n-k}\overline{\hat{\psi}(-n)}\overline{a}^{n+k}\overline{b}^{p-n-k}. 
	\end{equation}
	
	Conversely, let us assume that for each $k, p\in \mathbb{N}\cup\{0\}$, \eqref{A1} holds. Then evaluating the expression $(T_\psi C_\phi-(T_\psi C_\phi)^*)z^k$ for each $k\in \mathbb{N}\cup\{0\}$ gives the value as zero.
	Hence, we obtain that the operator $T_\psi C_\phi$ is Hermitian on $\mathcal{H}^2$.
\end{proof}

	\end{document}